\documentclass{amsart}
\usepackage{amsmath, amsthm, amssymb}
\usepackage{amssymb}
\usepackage{amsfonts}
\usepackage{amscd}
\usepackage{eufrak}
\usepackage{mathrsfs}
\usepackage{latexsym}
\usepackage{amsmath}
\usepackage{txfonts}

\usepackage{tikz}
\usepackage{graphicx}
\usepackage{subfig}

\usepackage[all, cmtip]{xy}

\newtheorem{theorem}{Theorem}[section]
\newtheorem{lemma}[theorem]{Lemma}

\newtheorem{proposition}[theorem]{Proposition}

\theoremstyle{definition}

\theoremstyle{remark}
\newtheorem{remark}[theorem]{Remark}

\newcommand{\supp}[1]{\textrm{supp}(#1)}

\newcommand{\spec}[1]{\textrm{spec}\,#1}
\newcommand{\sshf}[1]{\mathcal{O}_{#1}}
\newcommand{\shf}[1]{\mathcal{#1}}

\newcommand{\iso}{\simeq}
\newcommand{\ses}[3]{0\rightarrow#1\rightarrow#2\rightarrow#3\rightarrow{0}}
\newcommand{\embedding}{\hookrightarrow}

\newcommand{\is}[1]{\mathscr{I}_{#1}}

\newcommand{\paren}[1]{\left(#1\right)}

\newcommand{\Hilb}[2]{\text{Hilb}^{#1}(#2)}
\newcommand{\soc}[1]{\text{Soc}\paren{#1}}

\numberwithin{equation}{section}

\begin{document}
\allowdisplaybreaks
\title{On the universal family of Hilbert schemes of points on a surface}

% Information for first author
\author{Lei Song}

\address{Department of Mathematics and Statistics,
University of Illinois at Chicago, Chicago, IL 60607}
\curraddr{Department of Mathematics, University of Kansas, Lawrence, KS 66045}
\email{lsong@ku.edu}

%\author{Author Two}
%\address{Mathematical Research Section, School of Mathematical Sciences,
%Australian National University, Canberra ACT 2601, Australia}

%    General info

%\date{07/31/15}

\dedicatory{}


\begin{abstract}
For a smooth quasi-projective surface $X$ and an integer $n\ge 3$, we show that the universal family $Z^n$ over the Hilbert scheme $\Hilb{n}{X}$ of $n$ points has non $\mathbb{Q}$-Gorenstein, rational singularities, and that the Samuel multiplicity $\mu$ at a closed point on $Z^n$ can be computed in terms of the dimension of the socle. We also show that $\mu\le n$.
\end{abstract}

\maketitle

\section{Introduction}

Let $X$ be a smooth quasi-projective surface over an algebraically closed field $k$ of characteristic $0$. Let $\Hilb{n}{X}$ denote the Hilbert scheme of zero dimensional closed subschemes of $X$ of length $n$. Fogarty's fundamental result \cite{Fogarty68} claims that $\Hilb{n}{X}$ is a smooth, irreducible variety of dimension $2n$. When $X$ is complete, one may consider $\Hilb{n}{X}$ as a natural compactification of the space of $n$ unlabeled distinct points on $X$. For this reason, $\Hilb{n}{X}$ is a quite useful tool for dealing with infinitely near points on $X$.

The universal family $Z^n\subset\Hilb{n}{X}\times X$, with the induced projection $\pi: Z^n\rightarrow\Hilb{n}{X}$, is a finite flat cover of degree $n$ over $\Hilb{n}{X}$. The importance of $Z^n$ lies in its role in inductive approaches for the study of Hilbert schemes, e.g.~in deduction of the Picard scheme of $\Hilb{n}{X}$ \cite{Fogarty73}, and in calculation of the Nakajima constants \cite{ES}. While $Z^2$ is simply the blow up of $X\times X$ along the diagonal,  $Z^n$ is complicated in general. It is well known that $Z^n$ is irreducible, singular (except $n$=2), and Cohen-Macaulay. By \cite{Fogarty73}, $Z^n$ is also normal with $R_3$ condition. More precisely, we prove in this note

\begin{theorem}\label{theorem 1}
The universal family $Z^n$ is non $\mathbb{Q}$-Gorenstein, and has rational singularities. For a closed point $\zeta=(\xi, p)\in Z^n$, the Samuel multiplicity $\mu= {b_2+1\choose 2}$, where $b_2=b_2(\sshf{\xi, p})$ is the dimension of the socle of $\sshf{\xi, p}$.
\end{theorem}

We note that $b_2+1$ equals the minimal number of generators of the ideal $\is{\xi, p}\subset\sshf{X, p}$; see Lemma \ref{b_2}.

For $i>0$, let $V^i=\left\{(\xi, p)\in Z^n \;|\; b_2(\sshf{\xi, p})=i\;\right\}$. The locally closed subsets $V^i$ form a stratification of $Z^n$. A proposition due to Ellingsrud and Lehn \cite{EllingLehn} says that $\text{codim}(V^i, \Hilb{n}{X}\times X)\ge 2i$. This kind of codimension estimate is an ingredient in proving the irreducibility of $\Hilb{n}{X}$ and more generally, that of certain quot schemes. An immediate consequence of their estimate is that $b_2\le n+1$. In fact, we have

\begin{theorem}\label{bound}
\begin{equation*}
  b_2\le\left\lfloor\frac{\sqrt{1+8n}-1}{2}\right\rfloor,
\end{equation*}
and the bound is optimal. Consequently, one has $\mu\le n$.
\end{theorem}

Haiman \cite[Prop 3.5.3]{Haiman01} proves a similar result on multiplicity, namely $\dim_k\sshf{\xi, p}\ge{b_2+1\choose 2}$, which implies that $\mu\le n$. We need to point out that $\mu\neq\dim_k\sshf{\xi, p}$ in general.

{\it Acknowledgments}: I am grateful to Lawrence Ein and Steven Sam for valuable discussions.

\section{Preliminaries}
For a point $p\in X$, $\frak{m}_p$ denotes the maximal ideal of the local ring $\sshf{X, p}$ and $k(p)$ denotes the residue field. Let $\xi\subset X$ be a zero dimensional closed subscheme with the defining ideal $\is{\xi}$. The socle $\soc{\sshf{\xi, p}}$ is defined to be $\text{Hom}_{\sshf{X, p}}(k(p), \sshf{\xi, p})$. We denote the minimal number of generators of $\is{\xi, p}$ by $e(\is{\xi, p})$. By Nakayama lemma $e(\is{\xi, p})=\dim_{k(p)}{\is{\xi, p}\otimes k(p)}$.

Since $X$ is smooth and of dimension two, $\sshf{\xi, p}$ has a minimal free resolution
\begin{equation}\label{minimal resolution}
 0\rightarrow{\overset{b_2}{\oplus}\sshf{X, p}}\xrightarrow{\varphi}{\overset{b_1}\oplus \sshf{X, p}}\rightarrow{\sshf{X, p}}\rightarrow\sshf{\xi, p}\rightarrow 0,
\end{equation}
where obviously $b_2+1=b_1$, and
$ \varphi$ is represented by a matrix with entries in $\frak{m}_p$. The Hilbert-Burch theorem \cite[20.4]{Eisenbud95} says that all the $b_2\times b_2$ minors of $\varphi$ generate the ideal $\is{\xi, p}$. A standard computation of homological algebra shows
\begin{lemma}\label{b_2}
With the notations above, $\dim_{k(p)}\soc{\sshf{\xi, p}}=b_2=e(\is{\xi, p})-1$.\qed
\end{lemma}

We recall that the \textit{incidence correspondence Hilbert scheme} $\Hilb{n, n-1}{X}$ is defined as
\begin{equation*}
 \{(\xi, \xi')\; |\; \xi'\; \text{is a closed subscheme of }\; \xi\}\subset\Hilb{n}{X}\times\Hilb{n-1}{X}
\end{equation*}
For such $(\xi, \xi')$, one has a unique point $p$ which fits into the exact sequence
\begin{equation*}
    \ses{\is{\xi}}{\is{\xi'}}{k(p)}.
\end{equation*}
So naturally associated to $\Hilb{n, n-1}{X}$, there are two morphisms
$$\xymatrix{
&\Hilb{n, n-1}{X}\ar[dr]^{\phi}\ar[dl]_{\psi} &\\
Z^{n}&  & \Hilb{n-1}{X}\times X}$$
by $\psi: (\xi, \xi')\mapsto (\xi, p)$ and $\phi: (\xi, \xi')\mapsto (\xi', p)$.

We collect a few facts about $\Hilb{n, n-1}{X}$ that will be needed in the next section.
\begin{proposition}[\cite{Tikhomirov92}, \cite{Cheah96}]\label{smoothness of nested Hilbert scheme}
$\Hilb{n, n-1}{X}$ is a smooth irreducible variety of dimension $2n$.
\end{proposition}
\begin{proposition}[\cite{ES}]\label{proj of dualizing sheaf}\hfill
\begin{enumerate}
  \item  $\Hilb{n, n-1}{X}\iso\mathbb{P}(\omega_{Z^{n}})$, where $\omega_{Z^{n}}$ is the dualizing sheaf. In particular, $\Hilb{n, n-1}{X}$ is a resolution of singularities of $Z^{n}$.
  \item $\Hilb{n, n-1}{X}\iso\text{Bl}_{Z^{n-1}}(\Hilb{n-1}{X}\times X)\iso\mathbb{P}\paren{\is{Z^{n-1}\subseteq {\Hilb{n-1}{X}\times X}}}$.
  \item  For a closed point $(\xi, p)\in Z^n$, $\psi^{-1}((\xi, p))\iso\mathbb{P}^{b_2-1}_k$.
\end{enumerate}
\end{proposition}

\section{Proofs}
We denote the discriminant divisor of the finite map $\pi: Z^n\rightarrow\Hilb{n}{X}$ by $B^n$, which is a prime divisor parameterizing all non-reduced subschemes of $X$ of length $n$.

\begin{proof}[Proof of Theorem \ref{theorem 1}]
In the proof of \cite[Thm 7.6]{Fogarty73}, Fogarty shows that there are precisely two prime Weil divisors on $Z^n$ dominating $B^n$, which are described as
\begin{eqnarray*}
  E^n_1 &=& \overline{\{(\xi, p)\; | \; \xi\;\text{has type}\; (2, 1, \cdots, 1),\text{and}\; l_p(\xi)=2 \}}, \\
  E^n_2 &=& \overline{\{(\xi, p)\; | \; \xi\; \text{non reduced}, l_p(\xi)=1 \}},
\end{eqnarray*}
with the reduced induced closed subscheme structure. Moreover he proves that neither $E^n_1$ nor $E^n_2$ is $\mathbb{Q}$-Cartier.

Let $r_i$ be the ramification index of $E^n_i$ over $B^n$, and $K(\ast)$ the field of functions of $\ast$. From the identity
\begin{equation*}
  r_1[K(E^n_1): K(B^n)]+r_2[K(E^n_2): K(B^n)]=n,
\end{equation*}
and $[K(E^n_1): K(B^n)]=1, [K(E^n_2): K(B^n)]=n-2$, we see that $r_1=2$ and $r_2=1$.

Then Riemann-Hurwitz formula for finite maps claims that
\begin{equation}\label{RH}
  K_{Z^n}+(2c-1)E^n_1+cE^n_2=\pi^*(K_{\Hilb{n}{X}}+cB^n),
\end{equation}
for all $c\in\mathbb{Q}$. Thereby we deduce that $K_{Z^n}$ cannot be $\mathbb{Q}$-Cartier by setting $c=\frac{1}{2}$ or $0$.

It thus remains to show the rational singularities of $Z^n$ and compute the Samuel multiplicity.
Let $R$ be the local ring $\sshf{\Hilb{n}{X}\times X, \zeta}$. Since $Z^n$ is a codimension two Cohen-Macaulay subscheme of $\Hilb{n}{X}\times X$, one has the minimal free resolution of $\sshf{Z^n, \zeta}$
\begin{equation}\label{minimal free res.}
    0\rightarrow{\overset{b_2}{\oplus}R}\rightarrow{\overset{b_1}\oplus R}\rightarrow{R}\rightarrow\sshf{Z^n, \zeta}\rightarrow 0.
\end{equation}

\textbf{Claim A}: $b_2$ in (\ref{minimal free res.}) equals $b_2(\sshf{\xi, p})$.
\begin{proof}[Proof of Claim A]
Let $\{\xi\}\times{X}$ denote the fibre of $\Hilb{n}{X}\times X$ over $\xi$. Tensoring (\ref{minimal free res.}) with $\sshf{\{\xi\}\times{X}, p}$ yields the complex
\begin{equation}\label{complex}
     0\rightarrow{\overset{b_2}{\oplus}\sshf{X, p}}\rightarrow{\overset{b_1}\oplus \sshf{X, p}}\rightarrow{\sshf{X, p}}\rightarrow\sshf{\xi, p}\rightarrow 0.
\end{equation}
The complex is exact. In fact, for $i=1, 2$,
\begin{eqnarray*}
  \text{Tor}^{R}_{i}(\sshf{Z^n, \zeta}, \sshf{\{\xi\}\times{X}, p}) &\iso & \text{Tor}^{R}_{i}(\sshf{Z^n, \zeta}, R\otimes_{\sshf{\Hilb{n}{X}, \xi}}k(\xi))\\
   &\iso &  \text{Tor}^{\sshf{\Hilb{n}{X}, \xi}}_{i}(\sshf{Z^n, \zeta}, k(\xi))\\
   &=& 0,
\end{eqnarray*}
where the last step is by the flatness of $\sshf{Z^n, \zeta}$ over $\sshf{\Hilb{n}{X}, \xi}$.

Clearly $(\ref{complex})$ is minimal, so it is a minimal free resolution of $\sshf{\xi, p}$, and we are done.
\end{proof}

By taking dual $\text{Hom}_R(\cdot , R)$ for (\ref{minimal free res.}), one gets the exact sequence
\begin{equation*}
   {\overset{b_1}\oplus R}\rightarrow{\overset{b_2}{\oplus}R}\rightarrow{\text{Ext}^2_R(\sshf{Z^n,\zeta}, R)}\rightarrow 0.
\end{equation*}
Since $R$ is regular, one has $R\iso\omega_R$, the canonical module. By duality, $\text{Ext}^2_R(\sshf{Z^n, \zeta}, R)\iso \omega_{Z^n, \zeta}$, and hence $\mathbb{P}(\omega_{Z^n, \zeta})\embedding{\mathbb{P}:=\mathbb{P}^{b_2-1}_{R}}$. Writing $\rho$ for the projection $\mathbb{P}\rightarrow\spec{R}$, one has the commutative diagram
$$\xymatrix{
\mathbb{P}(\omega_{Z^n, \zeta}) \ar[d]_{\rho}\ar@{^{(}->}[r] & \mathbb{P}\ar[d]_{\rho}\\
 \spec{\sshf{Z^n, \zeta}}\ar@{^{(}->}[r] & \spec{R}}$$

\textbf{Claim B}: $\mathbb{P}(\omega_{Z^n, \zeta})$ is a local complete intersection in $\mathbb{P}$.
\begin{proof}[Proof of Claim B]
We have the diagram
$$\xymatrix{
\overset{b_1}\oplus\sshf{\mathbb{P}}  \ar[dr]^{\sigma} \ar[r] & \overset{b_2}\oplus\sshf{\mathbb{P}}\ar[d]\\
 & \sshf{\mathbb{P}}(1)}$$
where $\sshf{\mathbb{P}}(1)$ is the tautological invertible sheaf and the column map comes from the Euler sequence.

$\mathbb{P}(\omega_{Z^n, \zeta})$ is defined scheme-theoretically by the vanishing of $\sigma$, and hence has $b_1$ defining equations. On the other hand,
$\dim\mathbb{P}(\omega_{Z^n, \zeta})=2n=\dim\mathbb{P}-b_1$, i.e.~$\text{codim}(\mathbb{P}(\omega_{Z^n, \zeta}), \mathbb{P})=b_1$ as required.
\end{proof}
In view of Claim B, the Koszul complex
\begin{equation}\label{Koszul complex}
  0\rightarrow\bigwedge^{b_1}(\overset{b_1}\oplus\sshf{\mathbb{P}}(-1)) \rightarrow\cdots\rightarrow\bigwedge^2(\overset{b_1}\oplus\sshf{\mathbb{P}}(-1))\rightarrow {\overset{b_1}\oplus\sshf{\mathbb{P}}(-1)}\rightarrow\sshf{\mathbb{P}}\rightarrow\sshf{\mathbb{P}(\omega_{Z^n, \zeta})}\rightarrow 0
\end{equation}
is exact.

Observing that for all $j>0$ and $b_2-1\ge i\ge 0$,
\begin{equation*}
 R^j\rho_*\paren{\bigwedge^i\overset{b_1}\oplus\sshf{\mathbb{P}}(-1)}=R^j\rho_*\paren{\overset{b_1\choose i}\oplus\sshf{\mathbb{P}}(-i)}=0,
\end{equation*}
thus by splitting sequence (\ref{Koszul complex}) into short ones, we obtain that for all $j>0$,
\begin{equation*}
    R^j\rho_*\sshf{\mathbb{P}(\omega_{Z^n, \zeta})}=0.
\end{equation*}
Consider the fibred product for all closed point $\zeta\in Z^n$, cf.~Prop. \ref{proj of dualizing sheaf},
$$\xymatrix{
\mathbb{P}(\omega_{Z^n, \zeta})\ar[r]\ar[d]^{\rho}& \Hilb{n, n-1}{X}\ar[d]_{\psi}\\
 \spec{\sshf{Z^n, \zeta}}\ar[r]^{\nu} & Z^n.}$$
Since $\nu$ is flat,
\begin{equation*}
  R^j\psi_*\sshf{\Hilb{n, n-1}{X}}\otimes_{\sshf{Z^n}}\sshf{Z^n, \zeta}=0,
\end{equation*}
by base change theorem. It then follows that $Z^n$ has rational singularities, as $\Hilb{n, n-1}{X}$ is smooth, Prop. \ref{smoothness of nested Hilbert scheme}.

As for multiplicity, we use its relation with the Segre class $s_t(\cdot)$ and the invariance of Segre class under a birational morphism (cf. \cite[Chap. 4]{Fulton}), which give
\begin{equation*}
  \mu=s_0\paren{\{\zeta\}, Z^n}=s_{b_2-1}\paren{\mathbb{P}^{b_2-1}, \Hilb{n, n-1}{X}}=s_{b_2-1}(\shf{N}),
\end{equation*}
where $\shf{N}:=\shf{N}_{\mathbb{P}^{b_2-1}/{\Hilb{n, n-1}{X}}}$. We note that $\shf{N}\iso\shf{N}_{\mathbb{P}^{b_2-1}/\mathbb{P}(\omega_{Z^n, \zeta})}$.

By the exact sequence
\begin{equation*}
    \ses{\shf{N}}{\shf{N}_{\mathbb{P}^{b_2-1}/\mathbb{P}}}{\shf{N}_{\mathbb{P}(\omega_{Z^n, \zeta})/{\mathbb{P}}}\big|_{\mathbb{P}^{b_2-1}}},
\end{equation*}
and observing $\shf{N}_{\mathbb{P}^{b_2-1}/\mathbb{P}}=\overset{2n+2}{\oplus}\sshf{\mathbb{P}^{b_2-1}}$ and $\shf{N}_{\mathbb{P}(\omega_{Z^n, \zeta})/{\mathbb{P}}}\big|_{\mathbb{P}^{b_2-1}}\iso \overset{b_1}{\oplus}\sshf{\mathbb{P}^{b_2-1}}(1)$, we get
\begin{equation*}
    \mu=s_{b_2-1}(\shf{N})=c_{b_2-1}\paren{\overset{b_1}{\oplus}\sshf{\mathbb{P}^{b_2-1}}(1)}={b_1\choose b_2-1}={b_2+1\choose 2},
\end{equation*}
where $c$ is the Chern class.
\end{proof}

\begin{remark}
As a smooth surface varies, the universal families over Hilbert schemes of $n$ points easily patch together. More precisely, let $f: X\rightarrow C$ be a smooth, projective morphism from a 3-fold $X$ to a smooth curve $C$ over $k$. By \cite[Thm. 2.9]{Fogarty68}, the relative Hilbert scheme of $n$ points $\Hilb{n}{X/C}$ is smooth over $C$. The composite morphism $h: Z^n_{X/C}\rightarrow\Hilb{n}{X/C}\rightarrow C$ from the universal family is a flat morphism, so one gets a flat family of algebraic varieties with non $\mathbb{Q}$-Gorenstein and rational singularities.
\end{remark}

\begin{proof}[Proof of Theorem \ref{bound}]
We can assume that $\supp{\xi}=\{p\}$. Let $(A, \frak{m})$ denote the local ring $\sshf{X, p}$ and $I$ the defining ideal of $\xi$. Since $A/I$ is Artinian local $k$-algebra, one has $A/I\iso {\hat{A}}/{\hat{I}}$ as $k$-algebras, where $\hat{A}$, the $\frak{m}$-adic completion of $A$, is isomorphic to $k[\![x, y]\!]$.

Let $(R, \frak{n})$ be the local ring $\sshf{\mathbb{A}^2, O}$ of $\mathbb{A}^2$ at the origin. For $r\gg0$, one has the surjection $R/{{\frak{n}}^r}\iso\hat{A}/{\hat{\frak{m}}^r}\rightarrow\hat{A}/{\hat{I}}\rightarrow 0$, so there exists an ideal $J\subset R$ such that $R/J\iso {\hat{A}}/{\hat{I}}$ as $k$-algebras. Therefore it suffices to consider the case that $X=\mathbb{A}^2$ and $\supp{\xi}=\{O\}$.

Fix a monomial order $>$ for monomials in $k[x, y]$ and let $in_{>}I$ be the initial ideal of $I$. By \cite[Thm. 15.17, Prop. 15.16]{Eisenbud95}, there exists a flat family of length $n$ subschemes of $X$ over a curve such that the central fibre corresponds to $k[x, y]/in_{>}I$, and the other fibres correspond to rings isomorphic to $k[x, y]/I$. Because $b_2$ is an upper semi-continuous function on the set of closed points on $Z^n$, it suffices to treat the case when $I$ is a monomial ideal.
\begin{figure}[ht]\label{staircase}
        \centering
        \begin{tikzpicture}[scale=0.45]
            %gray grids
            \foreach \x/\y in {{1/4},{2/4}, {3/3}, {4/3}, {5/2}, {6/2}, {7/2}, {8/2}}
                \draw[gray] (\x, 0)--(\x,\y);
            \foreach \x/\y in {{1/5}, {1/4}, {3/3}, {5/2}, {9/1}}
                \draw[gray] (0, \y)--(\x,\y);
            %black boundaries
            \draw[black, thick] (0, 0)--(0, 6)--(1, 6)--(1, 4)--(3, 4)--(3, 3)--(5, 3)--(5, 2)--(9, 2)--(9, 0)--(0, 0);
            %blue nodes
            \foreach \x/\y in {{0/6}, {1/4}, {3/3}, {5/2}, {9/0}}
                \node[draw, circle, black, fill=black, minimum size=5pt, inner sep=0pt] at (\x, \y) {};
            %red nodes
            \foreach \x/\y in {{1/6}, {3/4}, {5/3}, {9/2}}
                \node[draw, circle, gray, fill=gray, minimum size=5pt, inner sep=0pt] at (\x, \y) {};

            %coordinate
            %\node[scale=0.7] at (5.6, 2.3) {$(5,2)$};
        \end{tikzpicture}
        \caption{}
    \end{figure}

In the monomial case, as shown in Figure 1, a set of minimal generators of $I$ corresponds to the black nodes, $n$ the number of boxes, and $b_2$ the number of gray nodes. Then it becomes clear that
$\frac{(b_2+1)b_2}{2}\le n$ and $\left\lfloor\frac{\sqrt{1+8n}-1}{2}\right\rfloor$ is the optimal upper bound for $b_2$.
\end{proof}

\bibliography{OntheUniversalFamilyofHilbertSchemesofPointsonaSurface}
\bibliographystyle{amsplain}
\end{document}